\documentclass[preprint,12pt]{elsarticle}

\usepackage{amsthm}
\usepackage{amssymb}
\usepackage{amsmath}
\usepackage{graphicx}
\usepackage{subfig}
\usepackage{float}
\usepackage{color}
\usepackage{complexity}
\usepackage[ruled,vlined,linesnumbered,noend]{algorithm2e}

\newtheorem{observation}{Observation}[section]

\newtheorem{theorem}{Theorem}[section]
\newtheorem{lemma}{Lemma}[section]
\newtheorem{remark}{Remark}[section]
\newtheorem{corollary}{Corollary}[section]

\journal{Discrete Applied Mathematics}

\begin{document}

\begin{frontmatter}
\tnotetext[mytitlenote]{This study was financed by CAPES - Finance Code 001 and CNPq.}
\cortext[cor1]{Corresponding author: Diego Amaro Ferraz}

\title{Filling some gaps on the edge coloring problem of split graphs}

\author{Fernanda Couto$^a$} 

\affiliation{organization={Federal Rural University of Rio de Janeiro},
            city={Nova Iguaçu},
            country={Brazil}}

\author{Diego Amaro Ferraz$^b$, Sulamita Klein$^b$} 

\affiliation{organization={Federal University of Rio de Janeiro},
            city={Rio de Janeiro},
            country={Brazil}}

\begin{abstract}
A split graph is a graph whose vertex set can be partitioned into a clique and an independent set. A connected graph $G$ is said to be $t$-admissible if admits a spanning tree in which the distance between any two adjacent vertices of $G$ is at most $t$. Given a graph $G$, determining  the smallest $t$ for which $G$ is $t$-admissible, i.e., the stretch index of $G$ denoted by $\sigma(G)$, is the goal of the {\sc $t$-admissibility problem}. Split graphs are $3$-admissible and can be partitioned into three subclasses: split graphs with $\sigma=1, 2 $ or $3$. In this work we consider such a partition while dealing with the problem of coloring the edges of a split graph. Vizing proved that any graph can have its edges colored with $\Delta$ or $\Delta+1$ colors, and thus can be classified as \emph{Class~1} or \emph{Class~2}, respectively. 
\textsc{The edge coloring problem} is open for split graphs in general. In previous results, we classified split graphs with $\sigma =2$ and in this paper we classify and provide an algorithm to color the edges of a subclass of split graphs with $\sigma=3$.

\end{abstract}


\begin{highlights}
\item We consider a special partition of split graphs 
\item We classify an infinity family of $(\sigma=3)$-split graphs as Class~1 graphs and provide a polynomial-time algorithm.  
\end{highlights}


\end{frontmatter}


\section{Introduction}\label{sec:introduction}

A proper $k$-coloring of a graph is an assignment of $k$ colors to its elements (vertices or edges) in which adjacent or incident elements receive distinct colors. In this case, $G$ is said to be $k$-colorable.
In coloring problems, the goal is usually to minimize the number of assigned colors, that is, to determine the smallest $k$ for which the graph is $k$-colorable. In this text we deal with the \textsc{edge coloring problem}. In this problem, the smallest $k$ for which the graph is $k$-colorable is called the \emph{chromatic index} and is denoted by $\chi'(G)$. It is easy to see that the maximum degree $\Delta$ of a graph $G$ is a lower bound for its chromatic index. Furthermore, Vizing~\cite{Vizing} proved that $\Delta+1$ is an upper bound for the chromatic index of any graph $G$, and thus, graphs can be classified either as \emph{Class~1}, if they are $\Delta$-colorable, or as \emph{Class~2}, otherwise. Therefore, this problem is also known as the {\sc classification problem}.\\ 
A graph $G$ is called a split graph if $V(G)$ can be partitioned into a clique $Q$ and an independent set $S$. Throughout this text we denote a split graph as $G=((Q,S),E)$ and we consider that $Q$ is a maximal clique. Although, the \textsc{edge coloring problem} is known to be \NP-complete in general (\cite{Holyer1981TheNO}, \cite{SANCHEZARROYO1989315}), this problem remains open when restricted to split graphs. There are several results in the literature related to the edge coloring of split graphs (\cite{Wilson}, \cite{Plantholt}, \cite{Chen}, \cite{Sheila}). Furthermore, until a previous work of ours\cite{couto2023new1}, split graphs have been studied in the context of edge coloring by considering some subclasses such as split-indifference graphs~\cite{CARMENORTIZ1998209}, split-comparability graphs~\cite{ComparabilityOrtiz} and split interval graphs~\cite{Gonzaga}. Since the goal is to fully classify split graphs with respect to the edge coloring problem, an interesting question that arises is: which subclasses are left to be studied in order to finish the fully classification of split graphs?\\
The {\sc t-admissibility problem}~\cite{Cail}, which is another quite challenging problem in general, is known to be polynomially time solvable for split graphs. This problem aims to determine the smallest $t$ such that a given graph $G$ admits a tree $t$-spanner, that is, a spanning tree $T$ in which the greatest distance between any pair of adjacent vertices of $G$ is at most $t$. If $G$ admits a tree $t$-spanner, then $G$ is said to be \emph{$t$-admissible} and $t$ is the \emph{stretch factor} associated to the tree. The smallest stretch factor among all spanning trees of $G$ is the \emph{stretch index} of $G$, denoted by $\sigma(G)$, or simply $\sigma$. We call a graph $G$ with $\sigma(G)=t$ a $(\sigma=t)$-graph. Split graphs are known to be $3$-admissible~\cite{Panda} and therefore the \textsc{$t$-admissibility problem} partitions the class into $3$ subclasses: ($\sigma=1$)-split graphs (bi-stars, i.e., trees with $n$ vertices and at least $n-2$ leaves), ($\sigma=2$)-split graphs or ($\sigma=3$)-split graphs (see Figure~\ref{particao_admissiblidade}). From the literature~\cite{zbMATH02614481}, we know how to color the edges of bi-stars and, in a previous work~\cite{couto2023new1}, we were able to color the edges of ($\sigma=2$)-split graphs. Thus, in order to fully classify the \textsc{edge coloring problem} for split graphs, we are left with the study of split graphs with $\sigma(G)=3$ (see Figure~\ref{diagramaCA}). In this work, we fill some gaps in the study of edge-coloring of split graphs by coloring the edges of a subclass of ($\sigma=3$)-split graphs.\\ 

\begin{figure}[H]
        \includegraphics[scale=0.20]{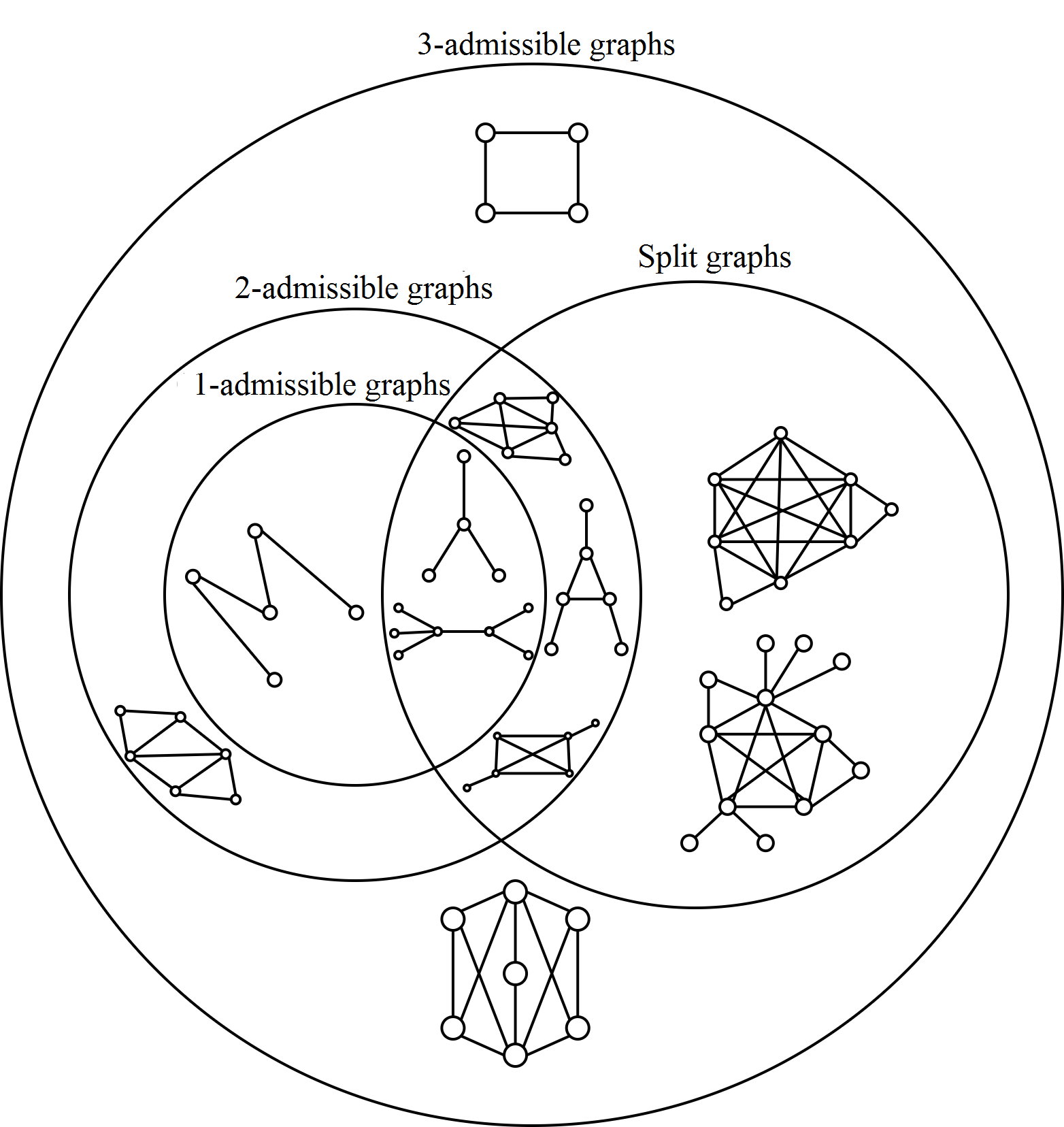}
        \centering
        \caption{Relation between split graphs and $3$-admissible graphs.   \label{particao_admissiblidade}}
    \end{figure}

\begin{figure}[H]
        \includegraphics[scale=0.50]{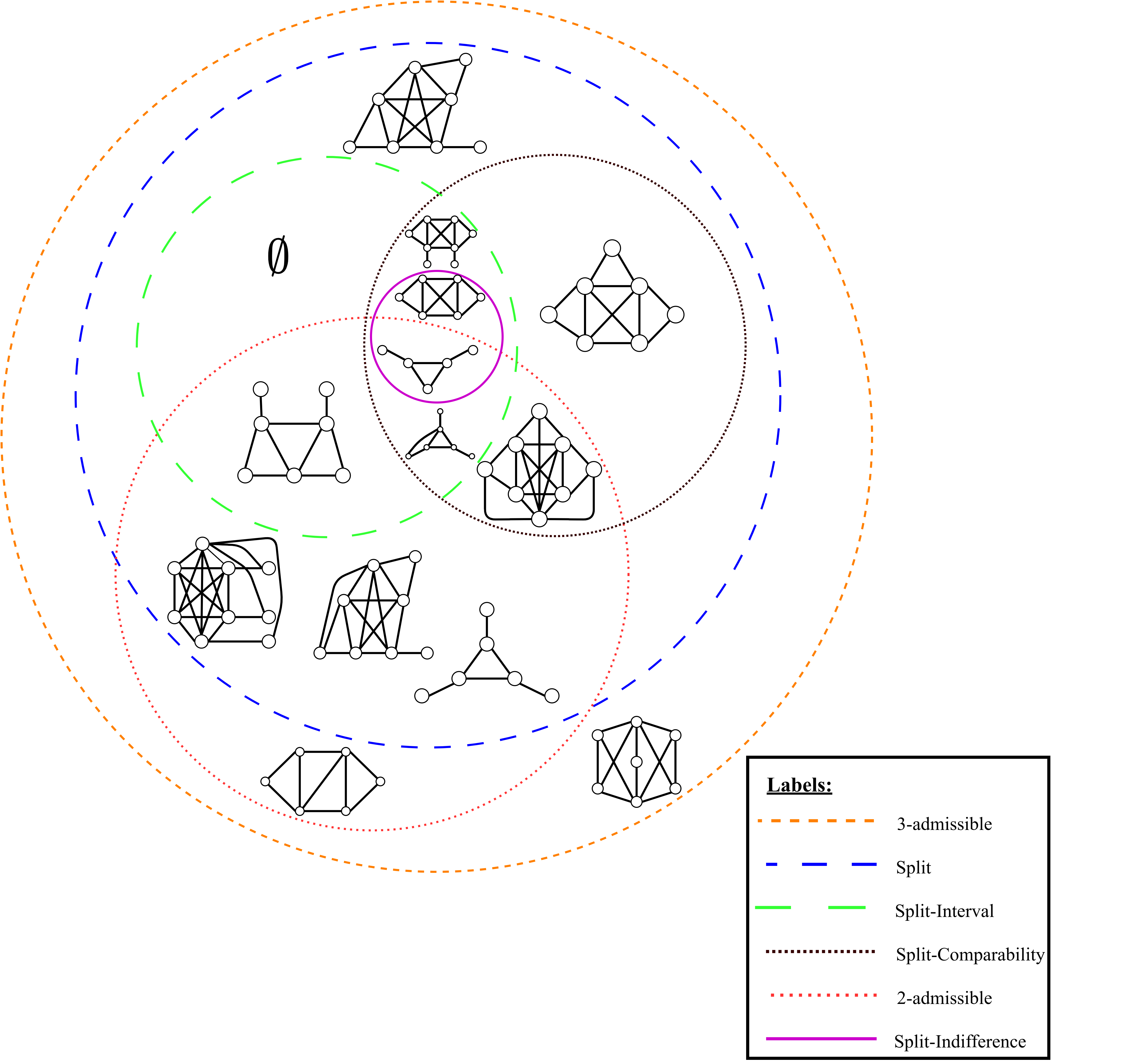}
        \caption{Some subclasses of split graphs and their relation with the $t$-admissibility problem. \label{diagramaCA}}
        \centering
    \end{figure}

This text is organized as follows: in Section~\ref{sec:introduction}, in addition to this brief introduction, we present some preliminary definitions which are used throughout this work; in Section~\ref{sec:edge-coloring}, we present an overview of the \textsc{edge coloring problem} and list some important results; in Section~\ref{sec:results}, we present our contributions. Finally, in Section~\ref{sec:conclusion} we summarize our results and present the final considerations and future guidelines for research.

\subsection{Preliminary Definitions}
In the following, we list some important definitions that are used in this work~\cite{Bondy1976}.

The \emph{degree} of a vertex $v$ is the number of edges incident to $v$ and is denoted by $d(v)$. If the degree of a vertex $v$ is equal to $1$, $v$ is called a \emph{pendant vertex}. The highest degree of a graph $G$ is denoted by $\Delta(G)$, whereas the smallest degree is denoted by $\delta(G)$. The vertex $v$ with degree equal do $\Delta$ (resp. $\delta$) is called a $\Delta$-vertex (resp. $\delta$-vertex). A graph $G$ is said to be \emph{overfull} if $\displaystyle{|E(G)|>\Delta(G)\cdot\left\lfloor{\frac{n}{2}}\right\rfloor}$, where $n=|V(G)|$. A graph $G$ is \emph{subgraph-overfull} if $\exists{H}\subseteq{G}$ such that $\Delta(H)=\Delta(G)$ and $H$ is overfull. A graph $G$ is \emph{neighborhood-overfull} if $H$ is induced by a vertex $v$ such that $d(v)=\Delta(H)$ and by all its neighbors. Let $C$ be the set of colors used on the edges of a vertex $v$. If a color $c \notin C$, then $c$ is called a \emph{missing color} of vertex $v$ and $c \in L(v)$, where $L(v)$ denotes the list of missing colors of the vertex $v$.

\subsection{The $t$-admissibility problem for split graphs}\label{sect:tadm_split}
The \textsc{$t$-admissibility problem} was introduced by Cai and Corneil~\cite{Cail} in 1995.  
The goal of this problem is to obtain a special spanning tree $T$ of a graph $G$ in which the greatest distance between any pair of adjacent vertices in $G$ is, at most, $t$. This spanning tree is called a \emph{tree $t$-spanner} and, if $G$ admits such a tree, the graph is said to be \emph{$t$-admissible} and $t$ is the \emph{stretch factor} associated with a tree $T$. The smallest value of $T$ for which $G$ is $t$-admissible is the \emph{stretch index} of the graph and is denoted by $\sigma(G)$, or simply $\sigma$. Deciding whether a graph $G$ is a $(\sigma=1)$-graph or a $(\sigma=2)$-graph are polynomial-time solvable problems. Indeed, it is easy to see that a graph is a $(\sigma=1)$-graph if, and only if, it is a tree. Moreover, Cai and Corneil~\cite{Cail} provided a linear-time algorithm to decide if a graph has $\sigma(G)=2$. They also settled that, for $t$ at least $4$, the problem is \NP-complete. Curiously, the problem is still open when the goal is to decide if a graph has $\sigma(G)=3$. However, there are some classes known to be $3$-admissible such as cographs~\cite{Couto} and split graphs~\cite{Panda}. In this work, we focus on split graphs. 

In particular, for split graphs, we have a characterization given by Theorem~\ref{thm:caracterizacao_split2adm}~\cite{Couto}, which is important for the results of this paper.\\ Initially, we have the following observation.

\begin{observation}
    \label{obs:pendant_vertex}
    Let $G=(V,E)$ be a graph and let $P=\{v\in{V}|d(v)=~1\}$. Then, $\sigma(G)=\sigma(G\setminus{P})$.
\end{observation}
 
From now on, we consider split graphs without pendant vertices.

Theorem~\ref{thm:caracterizacao_split2adm} is the most important theorem concerning the $t$-admissibility part to this work contributions.

\begin{theorem}[\cite{Couto}]
    Let $G=((Q,S),E)$ be a split graph such that $\forall{s}\in S, d(s)~>~1$. Then $\sigma(G)=2$ if, and only, $G$ has a universal vertex.
    \label{thm:caracterizacao_split2adm}
\end{theorem}

\section{The edge coloring problem}\label{sec:edge-coloring} 
A proper edge coloring of a graph consists in an assignment of colors to the edges of a graph $G=(V,E)$ such that edges incident to the same vertex are colored with different colors. From now on, any coloring we refer in the text is proper. The smallest number of colors needed to color the edges of a graph $G$ is called the \emph{chromatic index} of $G$ and is denoted by $\chi'(G)$. Note that it is easy to see that the maximum degree $\Delta$ of a graph $G$ is a lower bound for $\chi'(G)$ as we need $\Delta$ different colors to color the edges incident to a $\Delta$-vertex. In 1964, Vizing~\cite{Vizing} proved that $\Delta+1$ is an upper bound for the chromatic index of any graph $G$, and thus, graphs can be classified either as \emph{Class~1}, if they are $\Delta$-colorable, or  as \emph{Class~2}, otherwise. This problem is also known as the {\sc classification problem}. Surprisingly, even with only two possible values for $\chi'(G)$, the problem is proved to be \NP-complete~\cite{Holyer1981TheNO}, in general. However, it is open for several graph classes, for instance split graphs, the class we deal in this paper.


\subsection{Known results}

Although the \textsc{edge coloring problem} is open for split graphs, some partial results are known. In this subsection, we summarize some of these results.

\begin{theorem}[\cite{Behzad}]
    Let $G$ be a complete graph $K_n$. If $n$ is an even number, then $G$ is Class~1. Otherwise, $G$ is Class~2.
    \label{thm:complete-graphs}
\end{theorem}

Using an algorithm derived from Theorem~\ref{thm:complete-graphs}, we are able to color the edges of any complete graph. Furthermore, it is easy to see that any subgraph $H$ of an even complete graph $G$, such that $\Delta(H)=\Delta(G)$ is also Class~1 as stated in Theorem~\ref{thm:even-complete-graphs}.

\begin{theorem}[\cite{Behzad}]
    Let $G$ be a graph with an universal vertex and an even number of vertices. Then $G$ is Class~1.
    \label{thm:even-complete-graphs}
\end{theorem}

For graphs with an odd number of vertices and a universal vertex, we have the following theorem.

\begin{theorem}[Plantholt's Result \cite{Plantholt}]
    Let $G$ be a graph with an universal vertex and an odd number of vertices. Then $G$ is Class~2 if, and only if, $G$ is subgraph-overfull.
    \label{thm:plantholt}
\end{theorem}

The technique presented by Plantholt~\cite{Plantholt} colors with $\Delta(G)$ colors the edges of any graph $G$ (not overfull) with an odd number of vertices such that $\Delta(G)=|V(G)|-1$. The algorithm colors the edges of a graph $H^*$ with the same maximum degree and number of vertices of $G$, and exactly $\displaystyle{\Delta(G)\cdot\left\lfloor\frac{|V(G)|}{2}\right\rfloor}$ edges. We call $H^*$ a \emph{saturated graph}, as the addition of a single edge to $H^*$ transforms $H^*$ into an overfull graph. Note that $H$ can already be a saturated graph. The main idea of Plantholt's method is: if we can color the edges of $H^*$ with $\Delta(H)$ colors, then we can color the edges of any subgraph of $H^*$ (which maintains the number of vertices and the maximum degree of $H$) with $\Delta(H)$ colors, since removing edges cannot increase the chromatic index of a graph.
Chen et. al~\cite{Chen} provided an algorithm to color the edges of any split graph with an odd maximum degree.

\begin{theorem}[\cite{Chen}]
    Let $G$ be a split graph with an odd maximum degree. Then $G$ is Class~1.
    \label{thm:Chen}
\end{theorem}

Almeida et. al~\cite{Sheila} presented a sufficient condition for a split graph $G$ to be Class~1.

\begin{theorem}[\cite{Sheila}]
    Let $G=((Q,S),E)$ be a split graph. If exists some $v \in S$ s.t $\displaystyle{\left\lceil\frac{|Q|}{2}\right\rceil\leq d(v)\leq\frac{\Delta(G)}{2}}$, then $G$ is Class~1.
    \label{thm:sheila}
\end{theorem}

\subsection{($\sigma=2$)-Split graphs}\label{subsec:sigma2_split}

As mentioned before, Plantholt's method can color the edges of any graph, not overfull, with a universal vertex and an odd number of vertices. Since $(\sigma=2)$-split graphs (without pendants) have a universal vertex, it is possible to color the edges of any $(\sigma=2)$-split graph with an odd number of vertices using such a method. However, there is no specific way to obtain the saturated graph $H^*$ from a graph $H$, in general. 
Interestingly, the structure of some $(\sigma=2)$-split graphs allows us to use an specific algorithm to build a $H^*$.  
In general, there is just one property that the saturated graph must satisfy: its number of edges, which is exactly $\displaystyle{|E(H^*)|=\Delta(H^*)\cdot\left\lfloor{\frac{|V(H^*)|}{2}}\right\rfloor}$. The algorithm we present to construct $H^*$ can be applied to a $(\sigma=2)$-split graph $H$ in which $\delta(H) \leq \frac{|V(H)|-1}{2}$.

\begin{algorithm}[H]
    \caption{Construction of a special saturated graph \label{alg:construction}}
    \SetAlgoLined
    \KwData{
    A graph $H=((Q,S),E)$ with a vertex $s_1$, such that $d(s_1) \leq \frac{\Delta(H)}{2}$, and a universal vertex 
    }
    \If{$|Q|< |V(H)|-1$}{
        Add to $Q$ each vertex of $S$ but $s_1$, obtaining a clique $Q^*$ of size $|V(H)|-1$
        }
    \If {$|E^*| <\Delta(H^*)\cdot\left\lfloor{\frac{|V(H^*)|}{2}}\right\rfloor$, $E^* = E \cup E'$ where $E'$ are the edges added to obtain $Q^*$}{
        Add to $s_1$ the remaining edges 
        }
    \Return{$H^* = ((Q^*,S^*=\{s_1\}), E^*)$}

\end{algorithm}


\begin{lemma}\label{lem:satG}
    If $H$ is $(\sigma=2)$-split graph, not overfull, without pendant vertices, with an odd number of vertices and with $\delta(H)=d(s_1) \leq \frac{|V(H)|-1}{2}$, then the saturated graph $H^*$ presented in Algorithm~\ref{alg:construction} can be built.
    
\end{lemma}
     
\begin{proof}
    Since $H$ is $(\sigma=2)$-split graph without pendant vertices, $\Delta(H) = |V(H)|- 1 $. The saturated graph $H^*$, in this case, must have exactly $$(|V(H)|-1)\cdot \frac{|V(H)|-1}{2} = \frac{(|V(H)|-1)^2}{2}$$ edges (because $|V(H)|$ is odd). 
    If $H$ has one big clique of size $|V(H)|-1$, this clique has $$\frac{(|V(H)|-1)(|V(H)|-2)}{2} = \frac{|V(H)|^2-3|V(H)|+2}{2}$$ edges. This number subtracted from the total number of edges of $H^*$ is $\frac{|V(H)|-1}{2}$, i.e., $\frac{|V(H)|-1}{2}$ edges are needed to obtain a saturated graph. If $$d(s_1) \leq \frac{|V(H)|-1}{2},$$ it is possible to finish the construction of $H^*$. 
    
\end{proof}

\begin{remark}\label{rem:verticesofH*}
    Given a vertex $q \in Q(H^*)$, $\Delta - 1 \leq d_{H^*}(q) \leq \Delta$. Moreover, for the only vertex $s_1 \in S(H^*)$, $d_{H^*}(s_1)= \frac{\Delta}{2}$, which means that $s_1$ is adjacent to half of the clique in $H^*$.
\end{remark}

\section{A Polynomial-time algorithm}\label{sec:results}

In this section, we propose a polynomial-time algorithm to color the edges of any $(\sigma=3)$-split graph with a vertex $s_1 \in S(G)$ such that $d(s_1) \leq \frac{|V(G)|-1}{2}$ and $s_1$ is adjacent to a $\Delta$-vertex of $G$. As mentioned before, split graphs can be partitioned into three parts according to the $t$-admissibility problem: $(\sigma=1)$, $(\sigma=2)$ and $(\sigma=3)$-split graphs. The problem is already solved and proved to be polynomial for $(\sigma=1)$ and $(\sigma=2)$ split graphs~\cite{couto2023new1}. Once again we highlight that pendant vertices can be removed from a split graph and their edges can be colored after the edge coloring of the remaining graph, as explained in~\cite{couto2023new1}. So, we consider $(\sigma=3)$-split graphs without pendant vertices and, by Theorem~\ref{thm:caracterizacao_split2adm},  without a universal vertex. Moreover, by Theorem~\ref{thm:Chen}, if $\Delta(G)$ is odd, then $G$ is Class~1. Thus, our input graphs have even maximum degree.

\begin{remark} \label{remark:subgraph}
    A $(\sigma=3)$-split graph $G$ has an induced subgraph $H$ such that $(\sigma(H)=2)$ and $\Delta(G)=\Delta(H)$.
\end{remark}
 Indeed, any subgraph of $G$ induced by a $\Delta$-vertex and its neighborhood is a $(\sigma=2)$-split graph with the same $\Delta$. 
 
 Note that, by hypothesis, there is vertex $s_1 \in S(H)$, adjacent to some $\Delta$-vertex.
In the following, we present the main idea of Algorithm~\ref{alg:basic}.

  First, we determine $H$, a $(\sigma=2)$-split induced subgraph of $G$ (without pendants). Since $H$ has a universal vertex and $\Delta(H)$ is even, it can be edge-colored with Plantholt's Method using a saturated graph $H^*$ according to  Algorithm~\ref{alg:construction} described in Section~\ref{subsec:sigma2_split}. Next, for each edge in $E(H^* \setminus H)$, we update the list of missing colors of the vertices of $G$,  denoted by $L(v), \forall v\in V(G)$, with the colors associated to these edges of $E(H^* \setminus H)$. Note that the list of missing colors of vertices that are in $V(G\setminus H)$ are updated with $\Delta$ colors. Finally, we reobtain the graph $G$ which is now partially edge-colored. It remains to color $E(G\setminus H)$. The main idea here is to \emph{extend Plantholt's coloring} and to use the missing colors to finish the edge coloring of $G$.


 \paragraph{Extending Plantholt's Coloring}
 
 Remark~\ref{rem:2possibilities} follows from Algorithm~\ref{alg:construction}

     \begin{remark}\label{rem:2possibilities}
         For any vertex $x \in V(G), d_G(x)\leq d_{H^*}(x)+1$.
     \end{remark}

     Thus, given an edge $wx, w \in (V(G) \setminus V(H^*)$ there are only two possibilities: there is an edge of $E(H^* \setminus H)$ incident to $x$, say $xx'$, whose color can be attributed to $wx$; or not (it occurs iff $d_G(x) = d_{H^*}(x)+1$).
     

      

  \begin{lemma}\label{lem:missing-color}
    Let $wx \in E(G \setminus H)$ such that there is not an edge $xx'\in E(H^* \setminus H)$. Then there is only one color available to be assigned to $wx$.
\end{lemma}
\begin{proof}
    Indeed, note that this happens when $d_G(x)= \Delta$ and $d_{H^*}(x)=\Delta-1$. Since $H^*$ is $\Delta$-edge colored by Plantholt's coloring, there is only one missing color available to color $wx$. 
    
\end{proof}

\begin{lemma}\label{lem:distinct}
Given a saturated graph $H^*$ (according to Algorithm~\ref{alg:construction}), the missing colors of each $(\Delta-1)$-vertex of $H^*$ are mutually distinct.
\end{lemma}

\begin{proof}
    Let $q$ a $\Delta$-vertex of $H^*$ and $s$ be the vertex of $S(H^*)$. 
    Note that there are $\frac{\Delta}{2}$ vertices adjacent to $s$ and other $\frac{\Delta}{2}$ vertices not adjacent to $s$. We affirm that the color of $qs$ is a missing color of some $(\Delta-1)$-vertex of $H^*$. Suppose by contradiction that each $(\Delta-1)$-vertex has an incident edge colored with the color of $qs$. In this case, these $\frac{\Delta}{2}$ edges are incident to the $\Delta$-vertices of $H^*$ but $q$. Thus, there is a conflict of color at one of them, which leads to a contradiction.
    Similar arguments can be applied to prove that each color of the edges incident to $s$ in $H^*$ is a missing color of some $(\Delta-1)$-vertex of $H^*$. Moreover, by Remark~\ref{rem:verticesofH*}, each $(\Delta-1)$-vertex of $H^*$ has exactly one missing color, and since $H^*$ is properly edge colored, these colors are mutually distinct.

\end{proof}

 \begin{remark}\label{rem:difdelta}
     Let $v \in V(H)$. If $d_{H^*}(v) > d_{H}(v)$ and $d_G(v)= d_{H^*}(v)$, then there is a color assignment (not necessarily proper) to the edges of $E(G\setminus H)$ which are incident to $v$ provided by Plantholt's Method.
 \end{remark}

We denote the color assigned to an edge $e$ of $E(G\setminus H^*)$ through a correspondence with some edge $e'$ of $E(H^* \setminus H)$ colored using Plantholt's Method by \emph{color}$(e)_{P(e')}, e \in E(G\setminus H^*)$ and $ e' \in E(H^* \setminus H)$.

The main idea of extending Plantholt's edge coloring is: for each edge edge $wx \in E(G\setminus H)$, if there is an edge $xx'$ of $E(H^* \setminus H)$, $x,x' \in Q(G)$, then $wx$ is colored with the color of $xx'$. This extension can provoke some color conflicts concerning the colors incident to $w$. 
Finally, in order to finish the edge coloring of $G$, some missing colors must be applied to the edges in cases in which the degree of the clique vertex in $G$ is greater than the degree of this vertex in $H^*$. Each conflict, if any, is solved by a sequence of \emph{swap of colors}, according to Lemma~\ref{lem:troca}.

\paragraph{Providing a proper edge coloring} The process of extending Plantholt's coloring 
constructs an auxiliary structure $D^*$ called \emph{Color Trail}, which is a graph that can be briefly described as follows: 
  The Color Trail $D^*$ contains all edges incident to a vertex $w \in V(G \setminus H)$ with color conflicts, and these edges are placed from the left to the right beginning with the color with the greatest number of conflicting edges. Each edge $wx \in E(G \setminus H)$ is represented as a \emph{solid edge}. If there is an edge $xx' \in E(H^* \setminus H)$ and the color assigned to $xx'$ by Plantholt's algorithm is assigned to $wx$ by Algorithm~\ref{alg:basic}, then we say that $wx$ \emph{is paired} with $xx'$, which is represented in the Color Trail by a \emph{dashed line}. If the color of the edge $wx$ does not come from an edge of $E(H^* \setminus H)$, then in the Color Trail there is no dashed edge paired with $wx$. 
 There is only one rule to order the edges in the construction of the Color Trail: for each color, the rightmost edge is the one with no dashed pair, if one exists. Observe Figure~\ref{fig:WG1}.

\begin{algorithm}[H]
    \caption{Main Algorithm \label{alg:basic}}
    \SetAlgoLined
    \KwData{
    An edge coloring of $H^*$, the graph $G$ and the subgraph $H$.
    }
    \ForAll{$wx \in E(G\setminus H)$}{
    
        \If{$\Delta=d_G(x)>d_{H^*}(x)$}{
            Assign the missing color of $x$ to $wx$\\
            Remove $c(wx)$ from $L(w)$\\
        }
        \Else{
            \If{$d_G(x)=d_{H^*}(x)$}{
                $c_G(wx) \leftarrow c_{H^*}(xx'), xx'\in E(H^* \setminus H)$\\
                Remove $c(xx')$ from $L(w)$\\
            }
        
            \Else{
                \ForAll{$xx'\in E(H^* \setminus H)$}{
                    \If{$c(xx') \in L(w)$}{
                       $c_G(wx) \leftarrow c_{H^*}(xx')$ \\
                       Remove $c(xx')$ from $L(w)$\\
                    }
                    \textbf{break}\\
                    \Else{
                        $c_G(wx) \leftarrow c_{H^*}(xx')$, for some edge $xx' \in E(H^* \setminus H)$\\ 
                        Remove $c(xx')$ from $L(w)$\\
                        }
                    }
                }  
            }
    }
    \ForAll{$w \in V(G \setminus H)$}{ 
         \If{there is a conflict of colors at $w$} {
            Let $C=\{c_1,c_2, \ldots, c_j\},~j\geq 2$ be a non-increasing sequence of colors w.r.t the number of conflicts\\
            Construct the color trail $D^*$\\
            \ForAll{$i=2,\ldots, d(w)$}{
                Color-Swap($G,H^*,\mathcal{L},D^*,i$)\\
                \ForAll{$k=1, 2, \ldots, i-2 $}{
                    \ForAll{$j=i-1, \ldots, 1$}{
                        \If{$c_G(wx_i)= c_G(wx_j)$}{
                        Color-Swap($G, H^*,\mathcal{L},D^*,i$)
                        }
                    }
                    {\bf break}
                }
            }
        } 
    }  
    \Return{A $\Delta$-edge coloring of $G$}

\end{algorithm}

 \begin{lemma}\label{lem:troca}
     Let $w \in V(G\setminus H)$, $x,y \in Q(G)$, and suppose color$(wx)_{P(e')}=$ color$(wy)_{P(e'')}$. There is a swap of colors which leads to a proper edge coloring assignment.
 \end{lemma}

\begin{proof}
     First, note that if color$(wx)_{P(e')}=$ color$(wy)_{P(e'')}$,  then there are two different vertices of $H^*$, say $a,b$,  such that $ax = e'$ and $by = e''$ receive the same color assignment in $H^*$. Moreover, the edges $ay, bx \in E(G)$, otherwise, $wx$ and $wy$ would have been colored using the colors that were incident either to $a$ or either to $b$, and the edge coloring would be proper. Thus, it is possible to swap the colors of the edges $ay$ and $wy$, since the color used at $wy$ is a missing color of vertex $a$. Note that the color conflict at $w$ is solved once both colors were assigned to vertex $a$ by Plantholt's edge coloring.    
     
 \end{proof}

Figure~\ref{fig:color_swap} depicts a color swap.\\ 
 
 \begin{figure}[H]
        \includegraphics[scale=0.55]{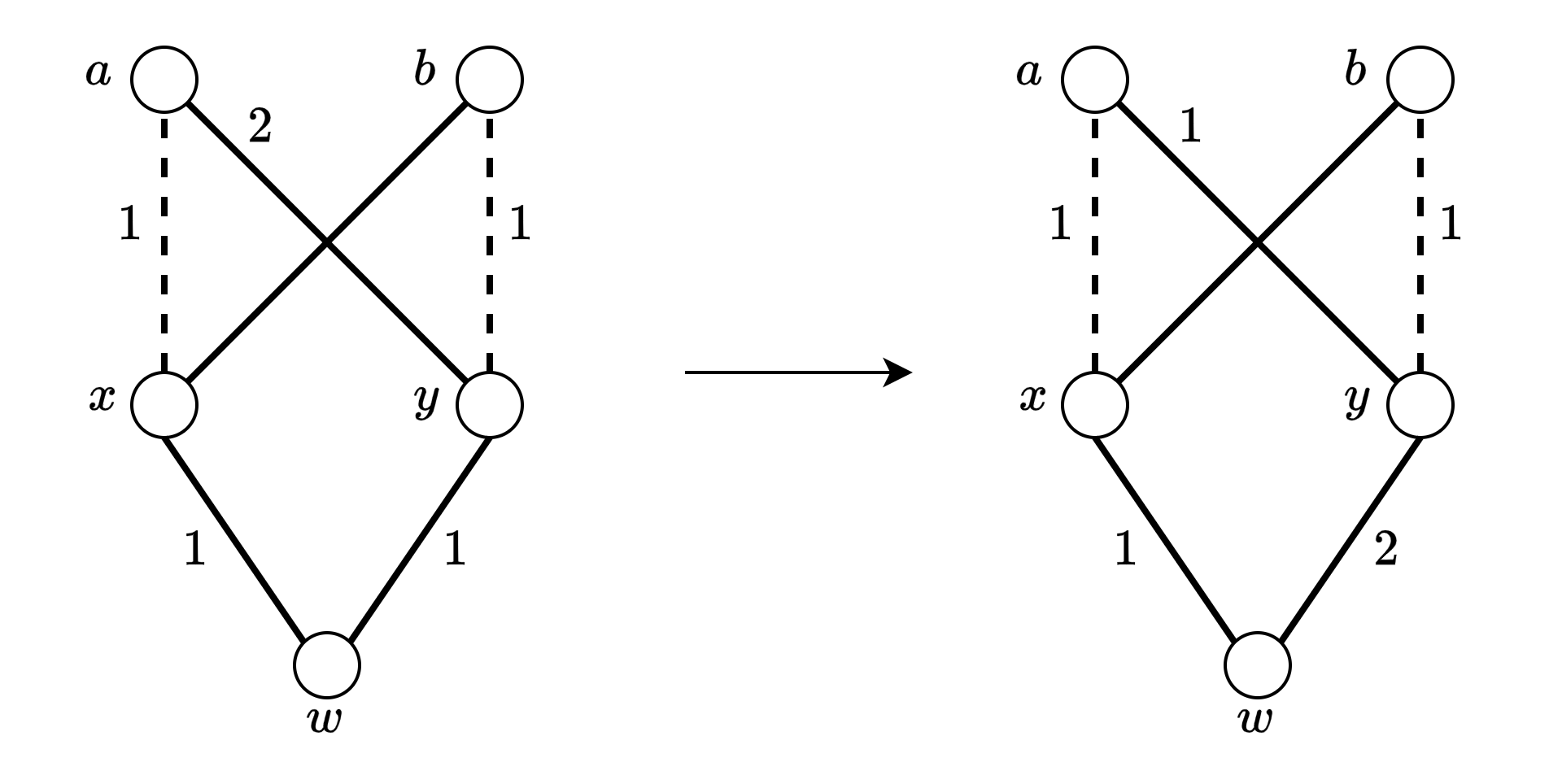}
        \centering
        \caption{On the left there are two conflicting edges: $wx$, $wy$ colored with color $1$. On the right, $wx$ remains colored with color $1$ and $wy$ and $ya$ swapped colors.}
        \label{fig:color_swap}
    \end{figure}

Note that once the colors of two edges are swapped, other conflict of colors can happen. These conflicts are treated by Procedure~\ref{proc:swap}. The idea is that the swaps occur linearly, from the left to the right considering the color trail. And the number of swaps for each edge has an upper bound according to the position of the edge in the Color Trail.

\begin{procedure}[H]

    \caption{Color-Swap($G,H^*\mathcal{L},D^*,i$) \label{proc:swap}}
    \SetAlgoLined
    \KwData{
    An edge coloring of $G$ (not proper), the subgraph $H^*$, a collection $\mathcal{L}$ of the lists of missing colors, the color trail $D^*$ and an index $i$.
    }
            \ForAll{j=i-1, \ldots, 1}{
                \If{$c_G(wx_i) \in L(x_j')$}{
                    $j^* \leftarrow j$
                }
            }
            $c_{\mbox{aux}} \leftarrow c_G(wx_i)$\\
            Remove $c_{\mbox{aux}}$ from $L(x_{j^*}')$\\
            Add $c_G(x_ix_{j^*}')$ to $L(x_{j^*}')$\\
            $c_G(wx_i)\leftarrow c_{H^*}(x_ix_{j^*}')$\\  
            $c_G(x_ix_{j^*}')\leftarrow c_{\mbox{aux}}$\\
            

\end{procedure}

\vspace{3mm}
Figures~\ref{fig:WG1},~\ref{fig:WG2},~\ref{fig:WG3} and~\ref{fig:WG4} show an application of Procedure~\ref{proc:swap} in resolving some color conflicts on a vertex $w$ of degree $5$. Note that in this example, only two colors are used on the edges incident to $w$. One of this colors is used three times and the other color is used two times, thus resulting in four color conflicts.

\begin{figure}[H]
  \centering
  \subfloat[]{\label{fig:WG1_a}\includegraphics[width=.45\textwidth]{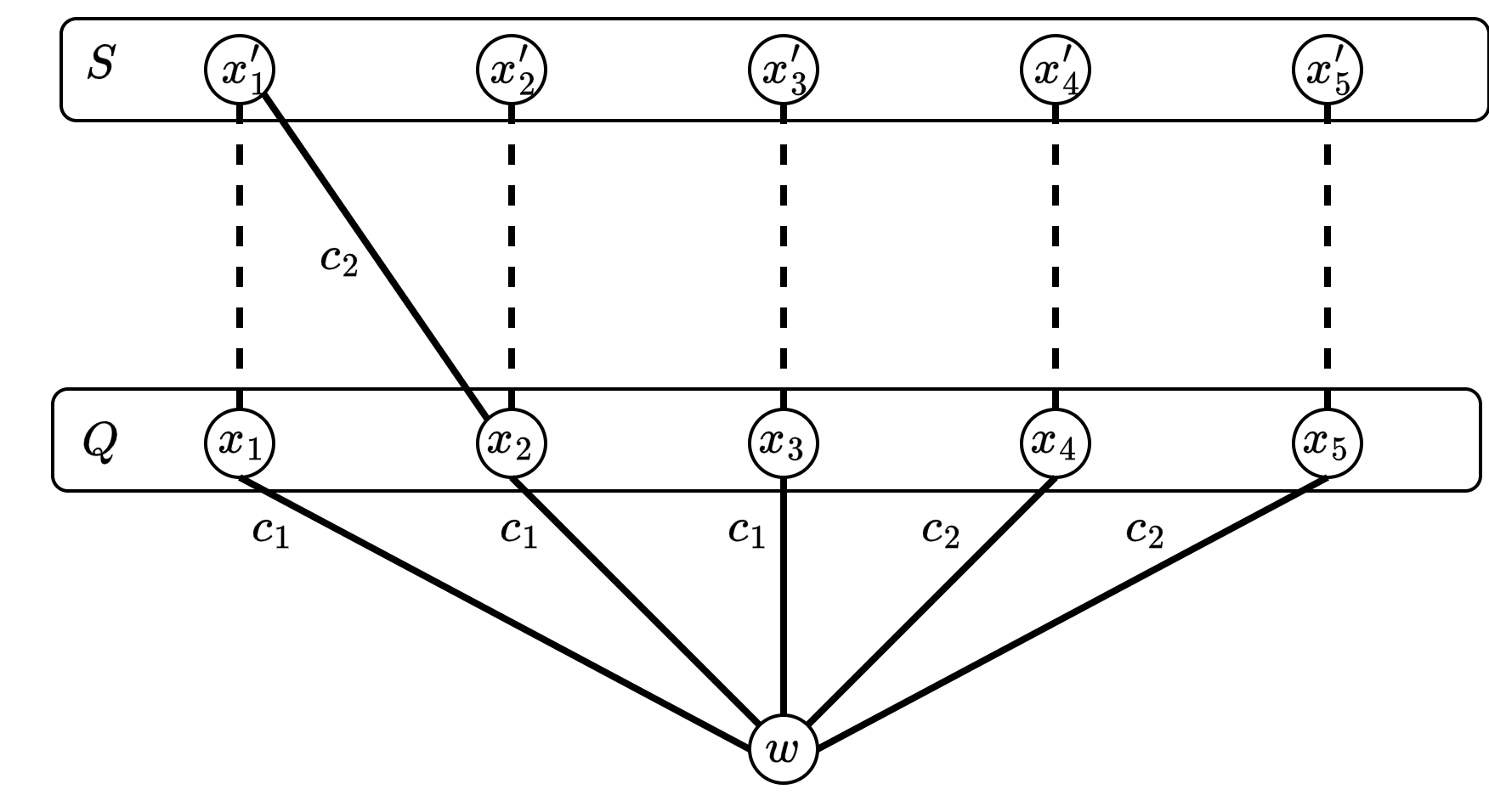}}\hspace{.7em}%
  \subfloat[]{\label{fig:WG1_b}\includegraphics[width=.45\textwidth]{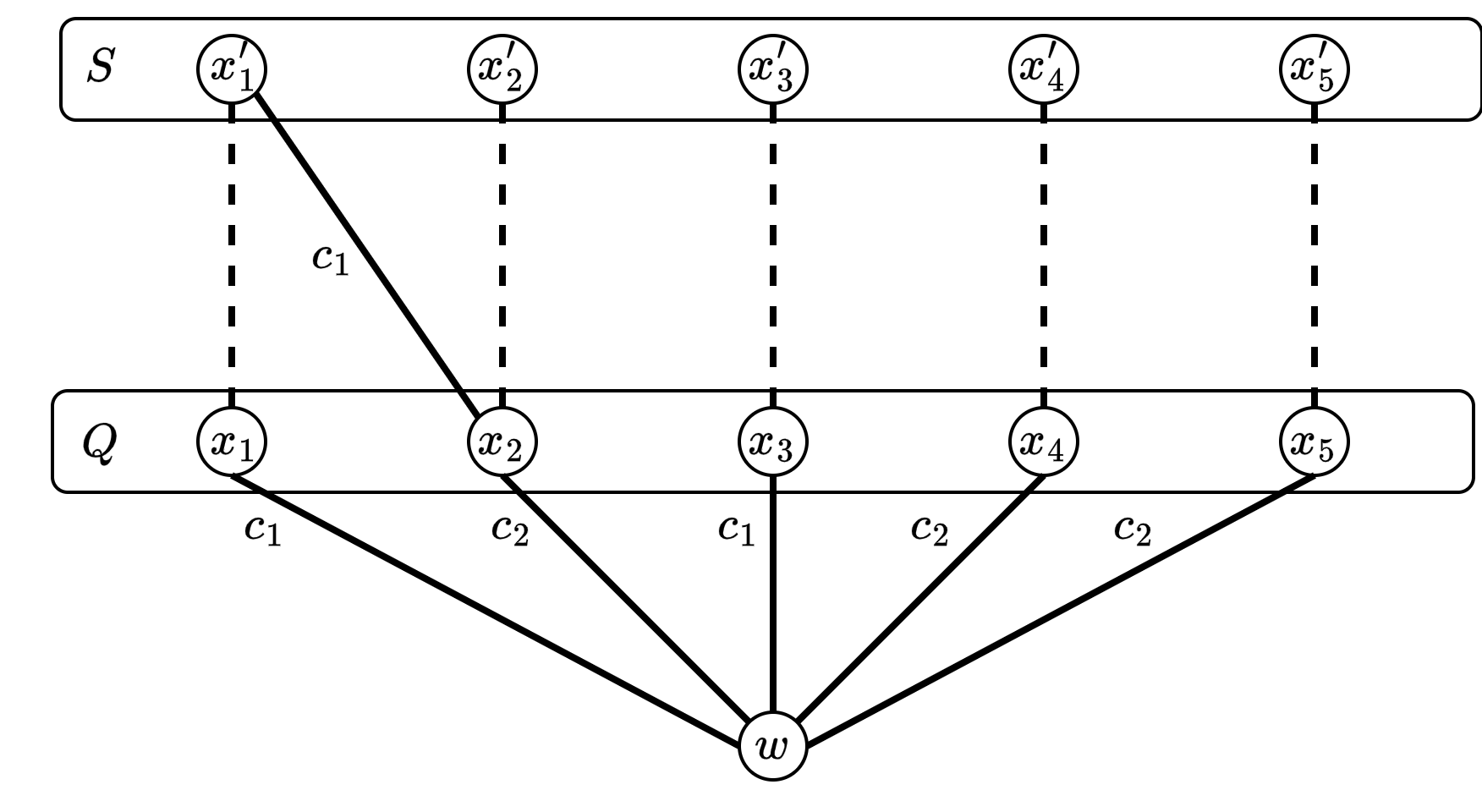}}
  \caption{(a) A Color Trail that shows two color conflicts on vertex $w$. (b) First swap of colors between edges $wx_2$ and $x_2x_1'$.}
  \label{fig:WG1}
\end{figure}

\begin{figure}[H]
        \includegraphics[scale=0.55]{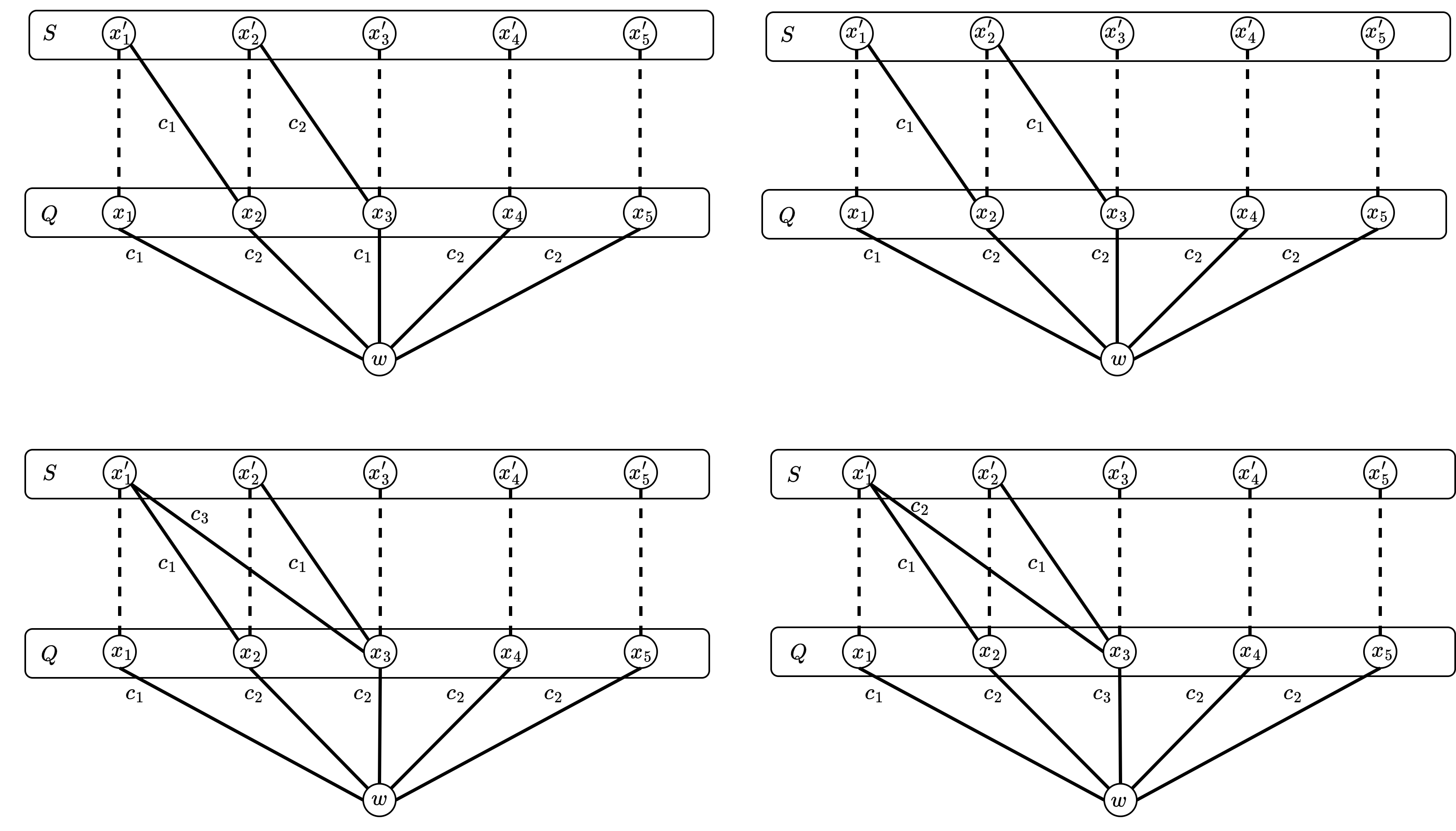}
        \centering
        \caption{The next conflict treated by the algorithm is the one involving the edges $wx_2$ and $wx_3$. Note that, for this second conflict, two swaps are needed.}
        \label{fig:WG2}
    \end{figure}

\begin{figure}[H]
        \includegraphics[scale=0.55]{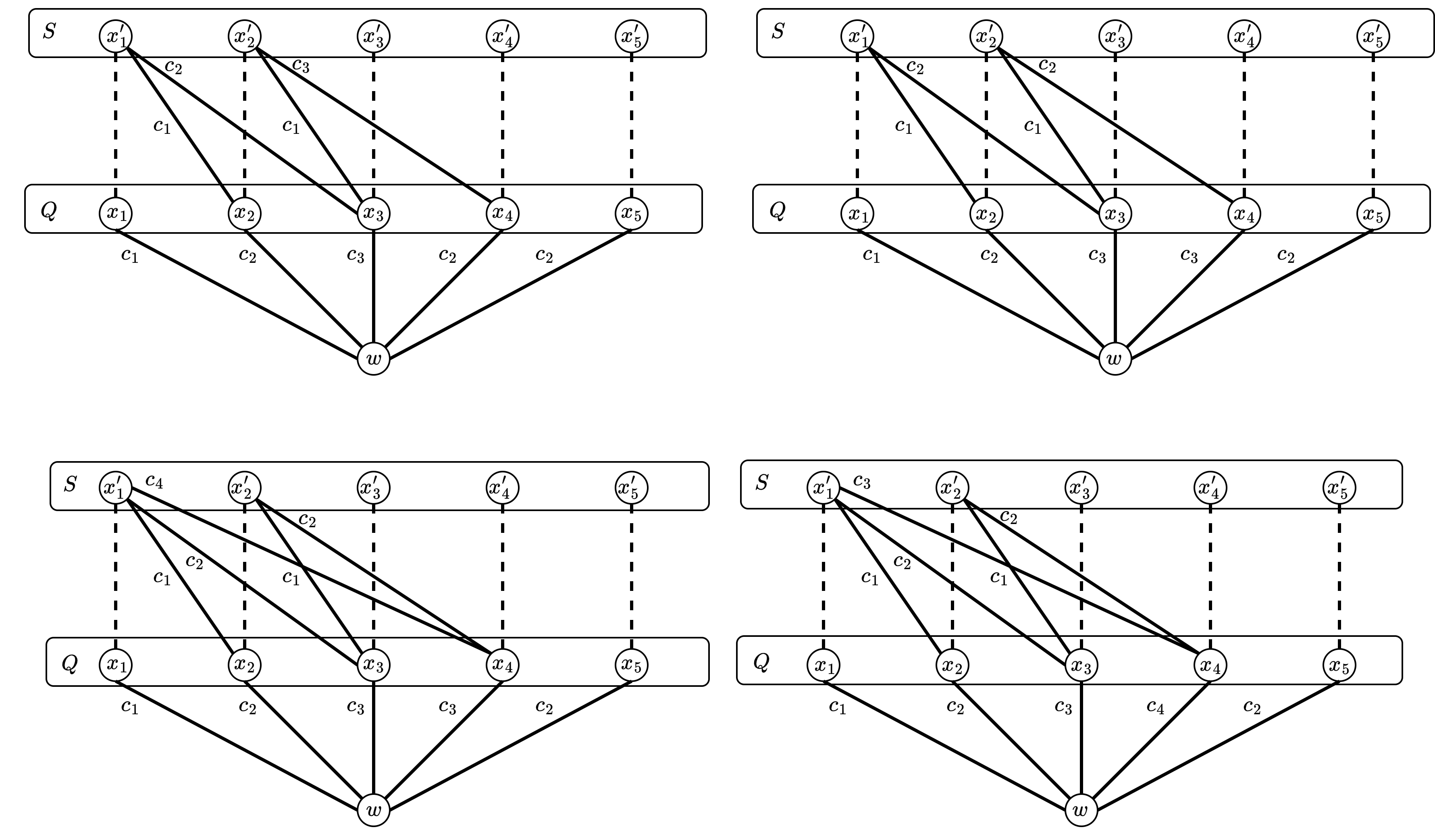}
        \centering
        \caption{The third conflict involves the edges $wx_3$ and $wx_4$. Just like the previous conflict, two swaps are necessary.}
        \label{fig:WG3}
    \end{figure}

\begin{figure}[H]
        \includegraphics[scale=0.55]{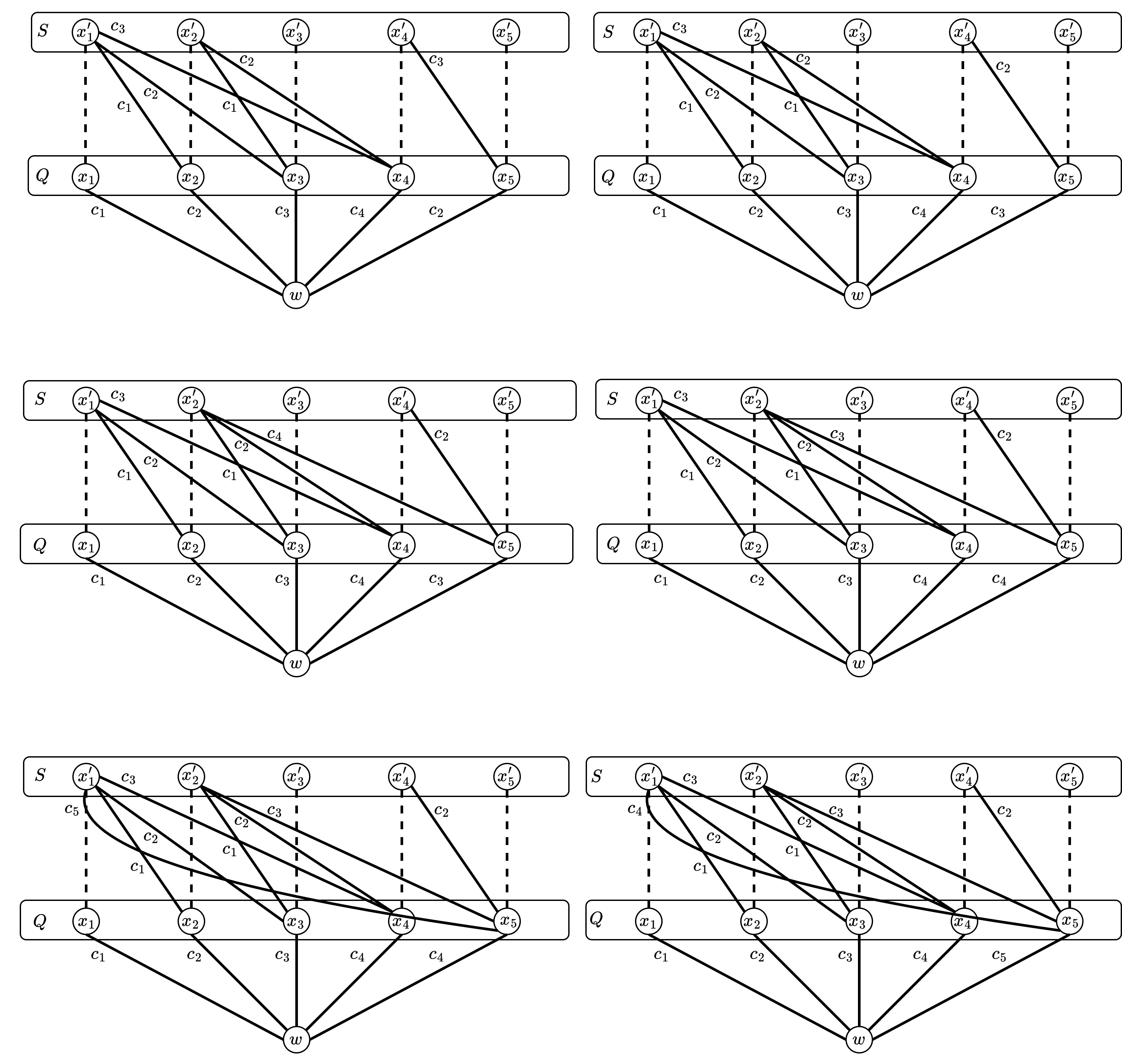}
        \centering
        \caption{The fourth and final color conflict needs three swaps to be resolved.}
        \label{fig:WG4}
    \end{figure}

\begin{lemma}\label{lem:colorexistence}
    If there is a color conflict between edges $wx_i$ and $wx_j$ there is a vertex $x_{j^*}'$ in the independent set of $G$ which has such a color as a missing color, for $1 \leq j^* < i$ in the color trail. 
\end{lemma}

\begin{proof}
Suppose there is a color conflict between edges $wx_i$ and $wx_j, j < i$. There are two cases to analyze. 

Case 1: $wx_j$ did not swap\\
In this case, $x_j$ is paired with $x_j'$ and $x_j'$ is a vertex in the independent set of $G$ with the color of $wx_j$ in its list of missing colors.

Case 2: $wx_j$ swapped at least once\\
Suppose $wx_j$ swapped with $x_j x_j^*, 1 \leq j^* < i$. Note that the color of $x_j x_j^*$ was assigned to $wx_j$ and, according to line $6$ of Procedure~\ref{proc:swap}, this color was added to the list of missing colors of $x_{j^*}'$.

\end{proof}

\begin{lemma}\label{lem:onlyswap}
    If there is a color conflict between edges $wx_i$ and $wx_j$, such that $i>j$, then the only edge incident to $w$ that has its color swapped is $wx_i$.
\end{lemma}

\begin{proof}
Note that by Procedure~\ref{proc:swap}, the swap of colors occurs between two edges: $wx_i$ and $x_ix_{j^*}'$ where $j^*$ is the position in the color trail of the leftmost vertex of the independent set of $G$ which has the conflicting color in its list of missing colors.
Therefore, the only edge incident to $w$ that has its color changed is $wx_i$, i.e., none of the edges $wx_l, 1\leq l \leq i-1$ have its color changed.

\end{proof}

\begin{lemma}\label{lem:limited-swaps}
    Let $w \in V(G \setminus H)$. For each edge incident to $w$ at position $i$ from the left to the right in the color trail, $1\leq i \leq d(w)$, at most $i-1$ swaps are needed.
\end{lemma}

\begin{proof}

A color conflict consists of two edges incident to the vertex $w$: the one in the $i$-th position of the color trail and another one on its left.  Once a swap is made, other color conflicts might happen. If the conflict is with a color assigned to an edge on the left of $wx_i$, then $wx_i$ is swapped. Since there are $i-1$ edges on the left of $wx_i$, at most $i-1$ swaps are made. Otherwise, by Lemma~\ref{lem:onlyswap}, the color of $wx_i$ does not change.

\end{proof} 

Corollary~\ref{cor:numberofswaps} follows from Lemma~\ref{lem:limited-swaps}.

\begin{corollary}\label{cor:numberofswaps}
For each $w \in V(G \setminus H)$ the number of swaps is $O(d^2(w))$.
\end{corollary}

\begin{lemma}\label{lem:corcolorswap}
    Procedure Color-Swap is correct and runs in polynomial time.
\end{lemma}

\begin{proof}
    The correctness of Procedure~\ref{proc:swap} follows from Lemmas~\ref{lem:troca},~\ref{lem:colorexistence}, \ref{lem:onlyswap} and \ref{lem:limited-swaps}. Moreover, it is easy to see that the procedure runs in $O(n^2)$-time.

\end{proof}

\begin{lemma}\label{lem:coralg1}
    The Main Algorithm (Algorithm~\ref{alg:basic})  is correct and runs in polynomial time.
\end{lemma}

\begin{proof}
    The correctness of Algorithm~\ref{alg:basic} follows from Remark~\ref{rem:2possibilities} and Lemmas~\ref{lem:satG},~\ref{lem:missing-color}, ~\ref{lem:distinct} and~\ref{lem:corcolorswap}. Moreover, it runs in $O(n^9)$-time.

\end{proof}

\begin{theorem}\label{thm:principal}
    Let $G$ be a $(\sigma=3)$-split graph. If there is a vertex $v \in S(G)$ adjacent to a $\Delta$-vertex and $\displaystyle{d(v)\leq\frac{|V(G)|-1}{2}}$, then $G$ is Class~1.
\end{theorem}

\begin{proof}
Since the graph $G$ is a $(\sigma=3)$-split graph which has a vertex $v\in S(G)$ adjacent to a $\Delta$-vertex and $\displaystyle{d(v)\leq\frac{|V(G)|-1}{2}}$, by Remark~\ref{remark:subgraph}, there is a  $(\sigma=2)$-split graph $H$ with maximum degree $\Delta$ that contains $v$. Thus, the saturated graph $H^*$ can be constructed following the steps presented in Algorithm~\ref{alg:construction}. Therefore, Algorithm~\ref{alg:basic}, which is correct by Lemma~\ref{lem:coralg1}, can be applied, and a $\Delta$-edge coloring of $G$ is given as an output.
    
\end{proof}

\section{Conclusion}\label{sec:conclusion}
   In this work, we classify as a Class~1 graph any $(\sigma=3)$-split graph $G$ with a vertex $v \in S(G)$ adjacent to a $\Delta$-vertex and $\displaystyle{d(v)\leq\frac{|V(G)|-1}{2}}$. Moreover we provide a polynomial-time algorithm, Algorithm~\ref{alg:basic}, that performs a $\Delta$-edge coloring for these graphs. We think some improvements can be done concerning the algorithm, so that its time complexity can be reduced. We are already working to extend the Main Algorithm (Algorithm~\ref{alg:basic}) to solve the problem for more $(\sigma=3)$-split graphs. These results bring us closer to fully classify split graphs with respect to the {\sc Edge Coloring Problem}.

\bibliographystyle{abbrv}
\bibliography{References}

\end{document}